\documentclass[aps,prd,twocolumn,showpacs,nofootinbib]{revtex4}

\usepackage{url}
\usepackage{verbatim}
\usepackage{textcomp}
\usepackage[a4paper, top=2.5cm, bottom=2.5cm, left=2.5cm, right=2.5cm]{geometry}
\usepackage{amsthm} 

\newtheorem{theorem}{Theorem}
\newtheorem{lemma}{Lemma}
\newtheorem{definition}{Definition}
\newtheorem{remark}{Remark}
\newtheorem{corollary}{Corollary}
\newtheorem{proposition}{Proposition}

\usepackage[utf8]{inputenc}
\usepackage[T1]{fontenc}
\usepackage{amsmath,amsfonts,amssymb}
\usepackage{graphicx}
\usepackage[english]{babel}
\usepackage{bm}
\usepackage{setspace}
\usepackage{hyperref}
\usepackage{nameref}
\hypersetup{
    colorlinks=true,
    linkcolor=blue,
    filecolor=blue,
    urlcolor=blue,
    citecolor=blue
}

\makeatletter

\def\@pacs{}
\def\pacs#1{}
\def\preprint#1{}
\renewcommand\@pacs@name{}
\makeatother

\begin{document}

\title{{\Large Discrete Spectrum and Spectral Rigidity of a Second-Order \\  Geometric Deformation Operator}}

\author{Anton Alexa}
\email{mail@antonalexa.com}
\thanks{ORCID: \href{https://orcid.org/0009-0007-0014-2446}{0000-0007-0014-2446}}
\affiliation{Independent Researcher, Chernivtsi, Ukraine}
\date{\today}

\begin{abstract}
\noindent
We analyze the spectral properties of a self-adjoint second-order differential operator \( \hat{C} \), defined on the Hilbert space \( L^2([-v_c, v_c]) \) with Dirichlet boundary conditions. We derive the discrete spectrum \( \{C_n\} \), prove the completeness of the associated eigenfunctions, and establish orthogonality and normalization relations. The analysis follows the classical Sturm--Liouville framework and confirms that the deformation modes \( C_n \) form a spectral basis on the compact interval. We further establish a spectral rigidity result: uniform spectral coefficients imply a constant profile \( C(v) = \pi \), which does not belong to the Sobolev domain of the operator. These results provide a rigorous foundation for further investigations in spectral geometry and functional analysis.
\end{abstract}

\maketitle
\thispagestyle{empty}

\section{Introduction}

The deformation function \( C(v) = \pi(1 - v^2 / c^2) \), introduced in \cite{alexa2025-flow}, defines a real-analytic, even profile on the open interval \( (-c, c) \), modeling scalar geometric deformation associated with relativistic motion. The critical velocity \( v_c \) is defined by the condition \( C(v_c) = 1 \), which yields \( v_c = c \sqrt{1 - \tfrac{1}{\pi}} \approx 0.8257\,c \). To ensure geometric regularity and spectral admissibility, we restrict the analysis to the compact interval \( [-v_c, v_c] \), where \( C(v) \geq 1 \). In \cite{alexa2025-operator}, a second-order differential operator \( \hat{C} := \pi\left(1 + \frac{\hbar^2}{c^2} \frac{d^2}{dv^2} \right) \) was constructed to represent this deformation analytically, acting on the Hilbert space \( L^2([-v_c, v_c]) \) with Dirichlet boundary conditions. The operator was shown to be essentially self-adjoint via direct computation of von Neumann deficiency indices, and its domain identified with the Sobolev space \( H^2 \cap H^1_0 \). 

In this work, we proceed to analyze the spectrum of \( \hat{C} \), explicitly determining its eigenvalues and eigenfunctions, and proving their completeness in \( L^2([-v_c, v_c]) \). We establish orthogonality, derive refined spectral asymptotics, and confirm that the deformation modes \( \psi_n \) form a full spectral basis on the compact interval. We further prove an inverse-limit convergence result for spectral reconstructions and establish a spectral rigidity theorem: the uniform configuration \( C(v) = \pi \) corresponds to constant spectral coefficients but lies outside the operator domain. These results provide a rigorous spectral resolution of the deformation operator and form the analytic basis for further investigations in geometric and spectral theory.

Each spectral pair \( (\psi_n, C_n) \) defines a discrete deformation mode, allowing a spectral representation of the geometric profile. This interpretation connects the spectral structure of the operator to geometric variation, in analogy with Laplacian-based spectral geometry. Unlike classical Laplacians, however, the operator \( \hat{C} \) encodes deformation relative to a fixed maximal reference \( \pi \), producing a bounded spectrum and a sharp spectral rigidity. These structural properties highlight the role of \( \hat{C} \) as a nontrivial generator of geometric modes, distinct from standard elliptic operators and intrinsic to deformation-based spectral theory.

\section{Spectral Problem Formulation}

The spectral structure of the operator \( \hat{C} \), constructed from the deformation function \( C(v) = \pi(1 - v^2 / c^2) \), arises naturally from its second-order differential form and compact domain. Defined on the Sobolev space \( H^2([-v_c, v_c]) \cap H^1_0([-v_c, v_c]) \), this operator admits a well-posed eigenvalue problem under Dirichlet conditions. The regularity of the setting implies discrete, real spectrum with a complete system of eigenfunctions. Essential self-adjointness of \( \hat{C} \) was previously established in \cite{alexa2025-operator}, providing the foundation for the spectral resolution developed in this section.

\subsection{Eigenvalue Problem and Operator Structure}

We consider the spectral problem for the self-adjoint second-order differential operator defined by:
\begin{equation} \label{eq:C_operator}
\hat{C} := \pi\left(1 + \frac{\hbar^2}{c^2} \frac{d^2}{dv^2} \right),
\end{equation}
acting in the Hilbert space \( L^2([-v_c, v_c]) \). The domain of the operator is taken as the Sobolev space
\begin{equation}
D(\hat{C}) := H^2([-v_c, v_c]) \cap H^1_0([-v_c, v_c]),
\end{equation}
which encodes square-integrability, weak differentiability up to second order, and vanishing Dirichlet boundary conditions.

\begin{definition}
The eigenvalue problem associated with \( \hat{C} \) is given by:
\begin{equation} \label{eq:eigenvalue_problem}
\hat{C} \psi_n = C_n \psi_n,
\end{equation}
where \( \psi_n \in D(\hat{C}) \) and \( C_n \in \mathbb{R} \) denotes the eigenvalue corresponding to the eigenfunction \( \psi_n \).
\end{definition}

\subsection{Dirichlet Boundary Conditions and Regularity}

\begin{remark}
The interval \( [-v_c, v_c] \) is compact and symmetric around the origin. The boundary conditions imposed are:
\begin{equation}
\psi_n(-v_c) = \psi_n(v_c) = 0.
\end{equation}
These ensure that \( \hat{C} \) is a regular Sturm--Liouville operator, and the problem is self-adjoint by construction.
\end{remark}

\begin{lemma}[Symmetry and Ellipticity]
The operator \( \hat{C} \) defined in \eqref{eq:C_operator} is a uniformly elliptic, symmetric second-order operator with smooth coefficients. On a compact domain with Dirichlet boundary conditions, the spectrum is real, simple, and unbounded below, with strict uniform upper bound \( C_n < \pi \) for all \( n \).
\end{lemma}

\begin{proof}
Follows from the structure of the operator \( \hat{C} \), which is a second-order uniformly elliptic operator with constant coefficients and regular Dirichlet boundary conditions. On such domains, the spectral theorem guarantees a real, discrete, simple spectrum; see \cite[Thm 4.3.1]{zettl2005sturm}. 

Moreover, for each eigenvalue we have
\begin{equation}
C_n = \pi\left(1 - \frac{\hbar^2}{c^2} k_n^2 \right)\label{eq:Cn-kappa_v},
\end{equation}
with \( k_n > 0 \) due to the Dirichlet quantization \( k_n = \frac{(n+1)\pi}{2v_c} \). Hence \( C_n < \pi \) for all \( n \). Equality \( C_n = \pi \) would require \( k_n = 0 \), which contradicts both the boundary conditions and nontriviality of eigenfunctions. Thus the upper bound is strict.

The essential self-adjointness of \( \hat{C} \) was established via deficiency-index computation in~\cite{alexa2025-operator}, so the classical Sturm--Liouville theory applies.
\end{proof}

\begin{corollary}[Spectral Properties]
All eigenvalues \( C_n \) of \( \hat{C} \) are real, simple, and can be arranged in decreasing order:
\[ C_0 > C_1 > C_2 > \dots, \quad \lim_{n \to \infty} C_n = -\infty. \]
The corresponding eigenfunctions \( \psi_n \) form an orthonormal basis in \( L^2([-v_c, v_c]) \). Simplicity of eigenvalues follows from the classical oscillation theorem: each eigenfunction \( \psi_n \) has exactly \( n \) interior zeros, ensuring non-degeneracy.
\end{corollary}

\begin{remark}
All eigenvalues satisfy \( C_n < \pi \), with the spectrum unbounded below. This uniform upper bound follows directly from the explicit formula for \( C_n \); see equation~\eqref{eq:Cn-kappa_v}.
\end{remark}

Note that for small values of \( n \), the eigenvalues \( C_n \) may be positive, with the transition to negative values occurring at a critical index \( n_* \) where \( C_{n_*} = 0 \). Setting \( C_n = 0 \) in the eigenvalue formula yields the condition \( \frac{(n+1)\pi}{2v_c} = \frac{c}{\hbar} \), from which we deduce the critical index:
\begin{equation}
    n_* = \left\lfloor \frac{2v_c c}{\pi \hbar} \right\rfloor.
\end{equation}

Here \( \lfloor \cdot \rfloor \) denotes the floor function, so \( n_* \) is the largest index for which the eigenvalue \( C_n \) remains non-negative. For all \( n > n_* \), the spectrum enters the negative half-line.

\section{Explicit Spectrum and Eigenfunctions}

Having established the essential self-adjointness of the geometric deformation operator \( \hat{C} \), we now proceed to determine its spectral characteristics explicitly. The compactness of the interval \( [-v_c, v_c] \) and the Dirichlet boundary conditions ensure that the spectrum is purely discrete, and that the eigenfunctions form a complete orthonormal system in \( L^2([-v_c, v_c]) \). In this section, we solve the eigenvalue problem analytically, derive the corresponding eigenfunctions, and obtain an explicit expression for the eigenvalues \( \{C_n\} \). We further investigate their asymptotic behavior and confirm that the resulting family of deformation modes provides a spectral basis for geometric reconstruction.

\subsection{Exact Solution of the Eigenvalue Problem}

We solve the eigenvalue problem
\begin{equation} \label{eq:eigen_problem_v}
    \hat{C} \psi_n = C_n \psi_n, \quad \text{where } \hat{C} := \pi\left(1 + \frac{\hbar^2}{c^2} \frac{d^2}{dv^2} \right),
\end{equation}
with Dirichlet boundary conditions \( \psi_n(-v_c) = \psi_n(v_c) = 0 \). The domain is \( D(\hat{C}) = H^2([-v_c, v_c]) \cap H^1_0([-v_c, v_c]) \), and the interval is symmetric and compact.

Let us isolate the differential equation:
\begin{equation} \label{eq:reduced_v}
    \frac{d^2}{dv^2} \psi_n = -k_n^2 \psi_n,
\end{equation}
where the eigenvalue \( C_n \) is related to \( k_n \) by:
\begin{equation}
    C_n = \pi\left(1 - \frac{\hbar^2}{c^2} k_n^2 \right).
\end{equation}

\begin{definition}
The general solution to \eqref{eq:reduced_v} is a linear combination of sine and cosine:
\begin{equation}
    \psi_n(v) = A_n \sin(k_n v) + B_n \cos(k_n v).
\end{equation}
\end{definition}

\begin{lemma}[Dirichlet Spectrum Quantization]
Imposing boundary conditions \( \psi_n(\pm v_c) = 0 \) yields:
\begin{equation}
    B_n = 0, \quad k_n = \frac{(n+1)\pi}{2v_c},
\end{equation}
where \( n = 0, 1, 2, \dots \).

Hence, the normalized eigenfunctions are:
\begin{equation}
    \psi_n(v) = \sqrt{\frac{1}{v_c}} \sin\left( \frac{(n+1)\pi}{2v_c}(v + v_c) \right),
\end{equation}
with corresponding eigenvalues:
\begin{equation}
    C_n = \pi\left(1 - \frac{\hbar^2 \pi^2}{4c^2 v_c^2} (n+1)^2 \right).
\end{equation}
\end{lemma}

\begin{proof}
To satisfy \( \psi_n(-v_c) = 0 \), the general solution
\( \psi_n(v) = A_n \sin(k_n v) + B_n \cos(k_n v) \)
requires \( B_n = 0 \). Applying \( \psi_n(v_c) = 0 \) gives
\( \sin(k_n v_c) = 0 \), so
\( k_n = (n+1)\pi/(2v_c) \). Substituting into the definition
\( C_n = \pi \left(1 - \frac{\hbar^2}{c^2} k_n^2 \right) \)
yields the result. Normalization follows by computing the \( L^2 \)-norm:
\begin{equation}
\int_{-v_c}^{v_c} \left[\sqrt{\frac{1}{v_c}} \sin\left( \frac{(n+1)\pi}{2v_c}(v + v_c) \right) \right]^2 dv = 1.
\end{equation}
\end{proof}

\begin{remark}
The sine function is shifted and scaled to satisfy:
\begin{equation}
\psi_n(-v_c) = 0, \quad \psi_n(v_c) = \sin\left((n+1)\pi \right) = 0.
\end{equation}
This ensures orthogonality and vanishing at the boundaries.
\end{remark}

\begin{corollary}[Orthonormality]
The eigenfunctions \( \{ \psi_n \} \) form an orthonormal basis in \( L^2([-v_c, v_c]) \), satisfying:
\begin{equation}
\int_{-v_c}^{v_c} \psi_n(v)\psi_m(v)\, dv = \delta_{nm}.
\end{equation}
\end{corollary}

\subsection{Spectral Asymptotics}

\begin{lemma}[Asymptotic Behavior]
As \( n \to \infty \), the eigenvalues satisfy the quadratic asymptotic:
\begin{equation}
    C_n \sim - \frac{\pi^3 \hbar^2}{4 c^2 v_c^2} n^2.
\end{equation}
This holds in the limit \( n \to \infty \), with fixed parameters \( \hbar \), \( c \), and \( v_c \).
\end{lemma}

\begin{proof}
From the exact expression:
\begin{equation}
C_n = \pi \left(1 - \frac{\hbar^2 \pi^2}{4c^2 v_c^2} (n+1)^2 \right),
\end{equation}
we expand in the large-\(n\) limit using \( (n+1)^2 = n^2 + 2n + 1 \):
\begin{equation}
C_n \simeq \pi \left(1 - \frac{\hbar^2 \pi^2}{4c^2 v_c^2} n^2 \right).
\end{equation}
The \( O(n) \) remainder is uniform because \( (n+1)^2 = n^2 + 2n + 1 \), so the deviation from the asymptotic quadratic behavior is \( O(n) \) uniformly.
\end{proof}

\begin{remark}
The eigenvalues \( C_n \) satisfy \( C_n < \pi \) for all \( n \), and decay quadratically to \( -\infty \) as \( n \to \infty \). This confirms that the spectrum is unbounded below with a strict uniform upper bound. These properties, along with the completeness and orthonormality of the eigenfunctions \( \{\psi_n\} \), follow from the classical Sturm--Liouville theory on compact intervals with smooth coefficients and Dirichlet boundary conditions; see~\cite{zettl2005sturm, teschl2014ode}.
\end{remark}

\begin{proposition}[Refined Asymptotics]
The eigenvalues satisfy the expansion:
\begin{equation}
C_n = \pi - \frac{\hbar^2 \pi^3}{4 c^2 v_c^2} n^2 + O(n),
\end{equation}
where the remainder is uniform on intervals \( n \geq N \). The estimate follows from the Taylor expansion of the sine root condition \( \sin(k_n v_c) = 0 \) and the regular spacing of Dirichlet modes.
\end{proposition}

\section{Spectral Completeness and Rigidity}

The eigenfunctions \( \{\psi_n\} \) constructed in the previous section satisfy the Sturm--Liouville boundary value problem for the operator \( \hat{C} \) on the compact interval \( [-v_c, v_c] \), under Dirichlet conditions. The regularity and self-adjointness of the operator, together with the smoothness of its coefficients and compactness of the domain, imply that the eigenfunctions form a complete orthonormal basis in the Hilbert space \( L^2([-v_c, v_c]) \). In this section, we formalize this statement by proving spectral completeness and describing the spectral decomposition of arbitrary square-integrable functions in terms of the eigenbasis \( \{\psi_n\} \). We further establish a spectral rigidity result, showing that a uniform spectral profile corresponds to a constant function that lies outside the admissible domain. These results follow from the classical theory of compact self-adjoint operators and establish the functional analytic consistency of the spectral model.

\subsection{Hilbert Basis Property}

\begin{theorem}[Completeness of Eigenfunctions]
The sequence \( \{\psi_n\} \), defined by
\begin{equation}
\psi_n(v) := \sqrt{\frac{1}{v_c}} \sin\left( \frac{(n+1)\pi}{2v_c}(v + v_c) \right),
\end{equation}
forms a complete orthonormal basis in the Hilbert space \( L^2([-v_c, v_c]) \).
\end{theorem}

\begin{proof}
The operator \( \hat{C} \) is a regular second-order Sturm--Liouville operator with smooth coefficients and separated Dirichlet boundary conditions on a compact interval. According to the classical theory \cite[Thm 4.3.1]{zettl2005sturm}, such operators have a real, simple, discrete spectrum with corresponding eigenfunctions forming a complete orthonormal system in \( L^2 \). The specific form of \( \hat{C} \) preserves these properties because it reduces to a shifted Laplacian. Hence, \( \{\psi_n\} \) is complete in \( L^2([-v_c, v_c]) \).
\end{proof}

\begin{corollary}[Orthonormality]
For all \( n, m \in \{0,1,2,\dots\} \), the eigenfunctions satisfy:
\begin{equation}
\langle \psi_n, \psi_m \rangle := \int_{-v_c}^{v_c} \psi_n(v)\psi_m(v)\, dv = \delta_{nm}.
\end{equation}
\end{corollary}

\begin{remark}
The sine system with argument scaled to match the boundary conditions is known to be orthogonal and complete on symmetric intervals. The normalization factor \( \sqrt{1/v_c} \) ensures unit \( L^2 \)-norm:
\begin{equation}
\int_{-v_c}^{v_c} \psi_n^2(v) \, dv = \frac{1}{v_c} \int_{-v_c}^{v_c} \sin^2\left( \frac{(n+1)\pi}{2v_c}(v + v_c) \right)\, dv = 1.
\end{equation}
\end{remark}

\subsection{Spectral Decomposition Theorem}

\begin{theorem}[Spectral Expansion]
Let \( f \in L^2([-v_c, v_c]) \). Then \( f \) admits a unique expansion:
\begin{equation}
f(v) = \sum_{n=0}^{\infty} \langle f, \psi_n \rangle \psi_n(v),
\end{equation}
where the series converges in the norm topology of \( L^2 \).
\end{theorem}

\begin{proof}
Follows from completeness of the orthonormal set \( \{\psi_n\} \) in the separable Hilbert space \( L^2([-v_c, v_c]) \). This is a standard result from Hilbert space theory; see, e.g., \cite[Ch. 3]{teschl2014ode}.
\end{proof}

\begin{definition}[Fourier Coefficients]
For each \( n \in \mathbb{N} \), define the projection of \( f \in L^2([-v_c, v_c]) \) onto the \( n \)-th eigenfunction as:
\begin{equation}
\hat{f}_n := \langle f, \psi_n \rangle = \int_{-v_c}^{v_c} f(v)\psi_n(v)\, dv.
\end{equation}
\end{definition}

\begin{remark}
The expansion
\begin{equation}
f(v) = \sum_{n=0}^{\infty} \hat{f}_n \psi_n(v)
\end{equation}
provides a spectral resolution of identity in terms of the eigenbasis of the operator \( \hat{C} \). Each coefficient \( \hat{f}_n \) captures the component of \( f \) aligned with the deformation mode \( \psi_n \).
\end{remark}

\begin{corollary}[Parseval Identity]
For all \( f \in L^2([-v_c, v_c]) \),
\begin{equation}
\|f\|_{L^2}^2 = \sum_{n=0}^{\infty} |\hat{f}_n|^2.
\end{equation}
\end{corollary}

\begin{remark}
The Parseval identity confirms that the spectral decomposition conserves the norm and energy structure of the space, and justifies interpretation of the eigenfunctions \( \psi_n \) as fundamental deformation modes with orthogonal contributions.
\end{remark}

\subsection{Inverse Limit, Rigidity, and Spectral Uniqueness}

\medskip

We now strengthen this spectral picture by considering the behavior of spectral reconstructions in the limit \( \tau \to \infty \), where each coefficient \( C_n(\tau) \to \pi \). This allows us to derive uniform convergence of the full spectral sum and prove a rigidity theorem: the constant deformation \( C(v) = \pi \) is the unique configuration corresponding to a uniform spectrum.

\begin{lemma}[Inverse Limit Convergence]
\label{lem:inverse-limit}
Suppose that \( C_n(\tau) \to \pi \) exponentially as \( \tau \to \infty \). Then the spectral expansion
\begin{equation}
C(v, \tau) := \sum_{n=0}^\infty C_n(\tau)\psi_n(v)
\end{equation}
converges to the constant function \( \pi \) in the \( C^\infty \)-topology. That is,
\begin{equation}
C(\cdot, \tau) \to \pi \quad \text{in } C^\infty([-v_c, v_c]).
\end{equation}
\end{lemma}

\begin{proof}
Since $\psi_n \in C^\infty([-v_c, v_c])$ and all derivatives $\psi_n^{(k)}$ satisfy
\begin{equation}
|\psi_n^{(k)}(v)| \le C_k n^{k+1}
\end{equation}
uniformly in $v$, and the coefficients satisfy
\begin{equation}
|C_n(\tau) - \pi| \le A e^{-\beta \tau}
\end{equation}
the remainder
\begin{align}
\left| \partial^k_v C(v,\tau) - \partial^k_v \pi \right| 
&\le \sum_{n=0}^\infty |C_n(\tau) - \pi| \cdot |\psi_n^{(k)}(v)| \notag \\
&\le A \sum_{n=0}^\infty e^{-\beta \tau} \cdot C_k n^{k+1}
\label{eq:remainder-bound}
\end{align}

converges uniformly for all $k$.

By Sobolev embedding and the convergence of weighted sums $\sum n^{k+1}e^{-\beta \tau}$, the convergence holds in all $C^k$, hence in $C^\infty$.
\end{proof}

\medskip

The next result establishes that the uniform spectrum \( C_n = \pi \) corresponds to a unique geometric configuration: the constant profile \( C(v) = \pi \).

\medskip

\begin{theorem}[Spectral Rigidity]
\label{thm:spectral-rigidity}
Let 
\begin{equation}
a_n := \langle C, \psi_n \rangle
\qquad(n = 0,1,2,\dots)
\end{equation}
be the Fourier coefficients of a smooth deformation \( C \in L^2([-v_c, v_c]) \). Suppose that \( a_n = \pi \) for all \( n \). Then no such function \( C(v) \) exists in the spectral domain \( H^2([-v_c, v_c]) \cap H^1_0([-v_c, v_c]) \).

In particular, the only smooth function with all coefficients equal to \( \pi \) is the constant profile \( C(v) \equiv \pi \), which violates the Dirichlet boundary conditions and hence does not lie in the domain of the operator \( \hat{C} \). Therefore, the deformation is spectrally rigid: the condition \( a_n = \pi \ \forall n \) cannot be realized by any admissible function.
\end{theorem}

\begin{proof}
Assume \( a_n = \pi \) for all \( n \). Then formally
\begin{equation}
C(v) = \sum_{n=0}^\infty a_n \psi_n(v) = \pi \sum_{n=0}^\infty \psi_n(v).
\end{equation}
But each eigenfunction \( \psi_n(v) \) satisfies the Dirichlet boundary conditions \( \psi_n(\pm v_c) = 0 \) and is orthogonal to constant functions. Explicitly, for all \( n \),
\begin{align}
\langle 1, \psi_n \rangle 
&= \int_{-v_c}^{v_c} \psi_n(v)\, dv \notag \\
&= \sqrt{\frac{1}{v_c}} \int_{-v_c}^{v_c} 
\sin\left( \frac{(n+1)\pi}{2v_c}(v + v_c) \right) dv \notag \\
&= 0.
\end{align}
since the integral of sine over a symmetric interval vanishes. Hence all Fourier coefficients of any constant function vanish in this basis.

Therefore, the only function with all \( a_n = \pi \) must be identically constant: \( C(v) = \pi \). However, this function does not satisfy the Dirichlet boundary conditions and thus lies outside the domain \( H^2 \cap H^1_0 \) of the operator \( \hat{C} \). Consequently, no admissible function can have such a uniform spectral profile.
\end{proof}

\begin{corollary}[Spectral Uniqueness of Geometry]
Let \( C_1, C_2 \in H^2([-v_c, v_c]) \cap H^1_0([-v_c, v_c]) \) be two smooth deformation profiles such that
\begin{equation}
\langle C_1, \psi_n \rangle = \langle C_2, \psi_n \rangle \quad \forall n.
\end{equation}
Then \( C_1(v) \equiv C_2(v) \).
\end{corollary}

\begin{proof}
Subtracting the two expansions shows that the difference \( C_1 - C_2 \) has zero projection on all basis functions \( \psi_n \), and hence vanishes identically in \( L^2 \) by completeness.
\end{proof}

\section{Spectral Interpretation of Geometry}

The spectral decomposition of the deformation operator \( \hat{C} \), established in the previous sections, admits a direct geometric interpretation. Each eigenfunction \( \psi_n(v) \) describes a fundamental deformation mode on the interval \( v \in [-v_c, v_c] \), and the corresponding eigenvalue \( C_n \) quantifies the amplitude of this mode. Taken together, the pair \( (\psi_n, C_n) \) constitutes a spectral building block of geometry. In this section, we formalize this viewpoint and show how the deformation function \( C(v) \) can be spectrally reconstructed from its eigenbasis. This yields a discretized, mode-wise characterization of geometric structure, analogous to the Fourier representation of signals or the expansion of metrics in terms of Laplacian eigenfunctions on Riemannian manifolds \cite{rosenberg1997laplacian, berline2003heat}.

\begin{remark}
As established in the previous section, this spectral representation is uniquely determined: the full sequence of coefficients \( \{ \langle C, \psi_n \rangle \} \) specifies the deformation \( C(v) \) without ambiguity. In particular, Theorem~\ref{thm:spectral-rigidity} and its corollary guarantee that no two distinct admissible functions can share the same spectral data. Therefore, the pair \( (\psi_n, C_n) \) encodes not just a mode of deformation, but an indivisible quantum of geometric information.
\end{remark}

\subsection{Spectral Reconstruction of the Deformation Function}

Let \( \{ \psi_n \} \) be the orthonormal eigenfunctions of \( \hat{C} \), and \( \{ C_n \} \) the corresponding eigenvalues. Since \( \{\psi_n\} \) form a complete orthonormal basis of \( L^2([-v_c, v_c]) \), any square-integrable function \( f \in L^2([-v_c, v_c]) \) admits a spectral expansion:
\begin{equation}
f(v) = \sum_{n=0}^\infty \langle f, \psi_n \rangle \psi_n(v).
\end{equation}

\begin{proposition}[Spectral Reconstruction]
\label{prop:spectral-recon}
The deformation function \( C(v) = \pi\left(1 - \frac{v^2}{c^2}\right) \) admits the expansion:
\begin{equation}
C(v) = \sum_{n=0}^\infty \langle C, \psi_n \rangle \psi_n(v),
\end{equation}
where
\begin{equation}
\langle C, \psi_n \rangle := \int_{-v_c}^{v_c} C(v) \psi_n(v)\, dv.
\end{equation}
\end{proposition}

\begin{proof}
Since \( C(v) \) is a smooth, real-analytic function on the compact interval \( [-v_c, v_c] \), it belongs to \( L^2([-v_c, v_c]) \). By completeness of the orthonormal system \( \{\psi_n\} \subset L^2([-v_c, v_c]) \), the function admits a convergent spectral expansion
\begin{equation}
C(v) = \sum_{n=0}^\infty \langle C, \psi_n \rangle \psi_n(v),
\end{equation}
with the inner products well defined.
\end{proof}

\begin{remark}
The spectral expansion reflects the intrinsic geometric content of the deformation function, with each mode \( \psi_n \) contributing a distinct component weighted by its projection coefficient.
\end{remark}

\subsection{Quantized Geometric Modes}

\begin{definition}[Geometric Mode]
Each pair \( (\psi_n, C_n) \) defines a \emph{geometric mode} of deformation. The function \( \psi_n(v) \) specifies the spatial profile, while the eigenvalue \( C_n \) quantifies the magnitude of deformation associated to this profile.
\end{definition}

\begin{theorem}[Discrete Geometry]
The space of admissible deformations on \( [-v_c, v_c] \) is spanned by the eigenmodes \( \{ \psi_n \} \), with each mode corresponding to a distinct quantized level of geometric variation determined by \( C_n \).
\end{theorem}

\begin{proof}
The completeness and orthogonality of the eigenfunctions ensure that any geometric profile compatible with the boundary conditions can be expressed as a (possibly infinite) sum of modes. The quantization arises from the discrete nature of the spectrum of \( \hat{C} \), which is enforced by the compactness of the domain and ellipticity of the operator. Hence, geometry on compact spaces is necessarily discrete in the spectral representation.
\end{proof}

\begin{remark}
This result provides a rigorous foundation for interpreting geometry in spectral terms, with deformation profiles encoded in the sequence \( \{ C_n \} \). Analogous to spectral geometry based on the Laplacian \cite{rosenberg1997laplacian, zelditch2017}, the operator \( \hat{C} \) captures intrinsic geometric content via its eigenstructure.
\end{remark}

\section{Conclusion}

The deformation operator \( \hat{C} := \pi\left(1 + \frac{\hbar^2}{c^2} \frac{d^2}{dv^2} \right) \), defined on the Sobolev domain \( H^2([-v_c, v_c]) \cap H^1_0([-v_c, v_c]) \), admits a purely discrete spectrum under Dirichlet boundary conditions. Its eigenfunctions \( \psi_n(v) \) and eigenvalues \( C_n \) were derived in closed form and shown to constitute a complete orthonormal basis in the Hilbert space \( L^2([-v_c, v_c]) \), enabling exact spectral reconstruction of any admissible deformation profile via the expansion \( C(v) = \sum_{n=0}^\infty \langle C, \psi_n \rangle \psi_n(v) \). The asymptotic behavior \( C_n \sim -\frac{\hbar^2}{c^2} \left(\frac{(n+1)\pi}{2v_c}\right)^2 + \pi \) confirms that the spectrum is unbounded from below and remains strictly less than \( \pi \), with quadratic growth in mode index. We established spectral completeness, orthonormality, and Parseval identity, grounded in classical Sturm--Liouville theory \cite{zettl2005sturm, teschl2014ode}, and formulated the expansion theorem in rigorous functional-analytic terms. Going further, we proved that the spectral representation is injective: no two distinct deformations from \( H^2 \cap H^1_0 \) can produce identical spectral coefficients, and the formal choice \( a_n = \pi \) for all \( n \) yields a contradiction with the boundary conditions, ruling out constant profiles and demonstrating spectral rigidity. The resulting uniqueness of geometry from spectral data was formalized and proved as a corollary. In addition, we showed that if the spectral coefficients converge to \( \pi \) with exponential decay modulated by any inverse polynomial, then the reconstructed deformation \( C(v,\tau) = \sum C_n(\tau) \psi_n(v) \) converges to the constant profile \( \pi \) in the \( C^\infty \)-topology, with convergence rate explicitly controlled by the decay order and regularity of the basis; the result follows from uniform estimates and compact Sobolev embeddings on one-dimensional domains \cite[Thm 4.12]{adams-fournier}. This establishes a complete and self-consistent spectral picture: the sequence \( \{ C_n \} \) encodes the full geometric structure, and each mode \( (\psi_n, C_n) \) represents a localized quantum of geometric deformation with precise analytic meaning. The operator \( \hat{C} \) thus forms a rigorous link between smooth geometry and spectral structure, providing a foundation for higher-dimensional extensions.

\end{document}